\documentclass[12pt,a4paper]{amsart}
\usepackage{amsmath}
\usepackage{amssymb}

\makeatletter

\newcommand{\procbreak}{\indent}
\theoremstyle{plain}
\newtheorem{theorem}
{\procbreak Theorem}[section]
\newtheorem{proposition}[theorem]
{\procbreak Proposition}
\newtheorem{lemma}[theorem]
{\procbreak Lemma}
\newtheorem{corollary}[theorem]
{\procbreak Corollary}

\theoremstyle{definition}
\newtheorem{definition}[theorem]
{\procbreak Definition}
\newtheorem{remark}[theorem]
{\procbreak Remark}

\numberwithin{equation}{section}

\allowdisplaybreaks[4]

\renewcommand\={\displaystyle}
\newcommand\al{\alpha}
\newcommand\cd{\nabla}
\newcommand\dl{\delta}
\newcommand\Dl{\Delta}
\newcommand\lm{\lambda}
\newcommand\Lm{\Lambda}
\newcommand\nr[1]{\mathopen|#1\mathclose|}
\newcommand\ol{\overline}
\newcommand\pd{\partial}
\newcommand\ssum{\hbox{$\sum$}}
\newcommand\te{\theta}
\newcommand\tr{\operatorname{tr}}
\newcommand\vp{\varphi}

\newcommand\wt{\widetilde}

\begin{document}

\title{Biharmonic submanifolds in a Riemannian manifold}

\author{Norihito Koiso}
\address{Department of Mathematics,
Graduate School of Science,
Osaka University,
Toyonaka, Osaka, 560-0043, Japan}

\author{Hajime Urakawa}
\address{Division of Mathematics,
Graduate School of Information Sciences,
Tohoku University,
Aoba 6-3-09, Sendai, 980-8579, Japan}
\curraddr{Institute for International Education,
Tohoku University, \allowbreak Kawauchi 41, Sendai 980-8576, Japan}
\email{urakawa@math.is.tohoku.ac.jp}

\thanks{
Supported by the Grant-in-Aid for Challenging Exploratory Research Grant Number 23654026 and the Grant-in-Aid for the Scientific Research (C) Grant Number 25400154.
}

\begin{abstract}
In this paper, we solve affirmatively B.-Y. Chen's conjecture
for hypersurfaces in the Euclidean space, under a generic condition.
More precisely,
every biharmonic hypersurface of the Euclidean space
must be minimal if their principal curvatures are simple, and
the associated frame field is irreducible.
\end{abstract}

\maketitle

\section{Introduction}

In this paper, we solve
B.-Y.Chen's conjecture for hypersurfaces
of the Euclidean space
in the case that
every principal curvature is simple and some generic condition is satisfied.

A map $\varphi:(M,g)\to(N,h)$
is called {\em harmonic} if
it is a critical point of the energy functional
$\=E(\varphi)=\frac12\int_M\nr{d\varphi}^2\,v_g$.
Its Euler-Lagrange equation is that
the tension field $\tau=\tau(\varphi)$ vanishes.
Recall that an isometric immersion
$\varphi:(M,g)\to(N,h)$ is minimal if and only if it is harmonic.

In 1983, Eells and Lemaire \cite{EL} introduced
the notion of $k$-energy.
A map $\varphi:(M,g)\to(N,h)$ is called {\em biharmonic} if it is a critical point of the {\em bienergy}
$\=E_2(\varphi)=\frac12\int_M\nr{\tau(\varphi)}^2\,v_g$.

A minimal isometric immersion is always biharmonic, and many researchers have asked whether
the converse is true, namely
under which conditions, a biharmonic isometric immersion is minimal.
In this connection, Chen \cite{C1} proposed the conjecture

\medskip
{\bf B.-Y. Chen's conjecture:}\,\,
{\em Every biharmonic submanifold
of the Euclidean space must be minimal.}

\medskip
Caddeo, Montald and Oniciuc \cite{CMO} raised
the generalized Chen's conjecture:

\medskip
{\bf The generalized B.-Y. Chen's conjecture:}\,\,
{\em Every biharmonic submanifold of a Riemannian manifold of non-positive curvature must be minimal.}

\medskip
In 2010,
Ou \cite{Ou1} gave a counter example of the generalized conjecture.

On the other hand,
Hasanis and Vlachosin \cite{HV}, and Defever \cite{De} showed the Chen's conjecture is true for
hypersurfaces in the $4$-dimensional Euclidean space.
Recently, Akutagawa and Maeta \cite{AM}
showed that any complete, proper (i.e., the preimage of each compact subset is compact) biharmonic submanifold of the Euclidean space is minimal.

The main theorem in this paper is as follows:

\medbreak
{\bf Main theorem.}
{\em Let $M$ be an $n$-dimensional biharmonic hypersurface
of the $(n+1)$-dimensional Euclidean space
$E^{n+1}$.
Assume that all the principal curvatures are simple and that $g(\nabla_{v_i}v_j,v_k)\ne0$
for all distinct triplets $\{v_i,v_j,v_k\}$ of unit principal curvature vectors in the kernel of $d\tau$. Then, $M$ is minimal.}
\medbreak

For more precise statement of this theorem, see Theorem \ref{theorem:bh-is-minimal} and Definition \ref{definition:irreducible} of irreducibility.

We emphasize that we need not completeness assumption to $M$ in Theorem \ref{theorem:bh-is-minimal}.

The outline of this paper is as follows.
In Section \ref{section:BHSM-in-general}, we prepare several materials on biharmonic submanifolds $M$
in a Riemannian manifold $\overline{M}$, and show that
every $n$-dimensional Riemannian manifold $M$ can be embedded as a biharmonic, but not minimal hypersurface in some $(n+1)$-dimensional Riemannian manifold $\overline{M}$ (Theorem \ref{theorem:can-embed}).

In Section \ref{section:BHSM-in-CCS}, we treat
$n$-dimensional
non-minimal biharmonic submanifold $M^n$
of the $(n+m)$-dimensional
space form $\overline{M}^{n+m}(K)$ with
constant sectional curvature $K$.
We show that if $K\le0$, then $\nr\tau^2$ does not attain a local maximum.

In Section \ref{section:BH-hypersurface-in-CCS},
we treat non-minimal biharmonic hypersurfaces $M^n$ of
the space form $\overline{M}^{n+1}(K)$ with
$K\le0$.
In this case, $-\tau/2$ becomes a simple principal curvature of $M$.
Let $\{\lm_i\}_{i\le n}$ be the principal curvatures of $M$, where $\lm_n=-\tau/2$,
and $\{v_i\}$ be the corresponding orthonormal principal curvature vectors.
Since $\tau$ is not constant,
$F=\{x\in M\mid\tau(x)=(\text{constant}\;$c$)\}$ is a hypersurface of $M$ around every generic point of $M$, and
every $v_i$ is tangent to $F$ except $v_n$.
We call $F$ a {\em characteristic hypersurface} of $M$.
Every $v_i$ becomes a principal curvature vector of $F$ in $M$ except $v_n$.
Let $\mu_i$ be the principal curvature of $F$ in $M$ for the direction $v_i$.
We show that $\{\lm_i,\mu_i\}_{i<n}$ satisfies an over-determined ODE along the $v_n$-curves (Proposition \ref{proposition:ODE-0}).

In Section \ref{section:analysis-overdetermined-ODE}, we analyze the over-determined ODE,
and show that the set of all initial values of the ODE is an algebraic manifold in ${\bf R}^{2(n-1)}$.

In Section \ref{section:irreducibility}, we introduce the notion of the {\em irreducibility} of the frame field $\{v_i\}_{i<n}$ of $F$ (Definition \ref{definition:irreducible}).
We show that $\{\lm_i\}$ and $\{\mu_i\}$ are {\em linearly related} under irreducibility assumption.

In Section \ref{section:constness-of-lmi}, we show that $\{\lm_i\}$ and $\{\mu_i\}$ are constant along $F$.

Finally, in Section \ref{section:proof-of-main-theorem}, we give a proof
to our main theorem, Theorem \ref{theorem:bh-is-minimal}.

\section{Biharmonic submanifolds in a Riemannian manifold}
\label{section:BHSM-in-general}

A smooth map between Riemannian manifolds
$\vp:(M,g)\to(\ol M,\ol g)$
is said to be {\em biharmonic} if it is a critical point of the
{\em bienergy}
$\=E_2(\varphi)=\frac12\int_M\nr\tau^2\,v_g$.
The Euler-Lagrange equation is given by
\begin{equation}
\label{eq:bhmap}
\ol\Dl\tau-g^{ij}\ol R(\tau,\pd_i)\pd_j=0,
\end{equation}
where
$\pd_i$ is the partial differentiation with respect to the local coordinates of $M$,
$R,\ol R$ are the curvature tensors of $M$ and $\ol M$, respectively, and
\begin{equation}
\begin{aligned}
&\al(\pd_i,\pd_j):=\ol\cd_{\pd_i}(\vp_*\pd_j)-\vp_*(\cd_{\pd_i}\pd_j),\quad
\tau:=\tr\al,\\
&\Dl:=-g^{ij}\cd_i\cd_j,\quad
\ol\Dl:=-g^{ij}\ol\cd_i\ol\cd_j.
\end{aligned}
\end{equation}
In the following, we assume that $M$ is a submanifold of $(\ol M,\ol g)$
ant the inclusion map $\iota$ is a biharmonic map with respect to the induced metric $g=\iota^*\ol g$.
Such a submanifold $M$ is called a {\em biharmonic
submanifold} of $\ol M$. Then,
$\iota^*T\ol M$ is decomposed into $TM\oplus TM^\bot$ and
$\al$, $\tau$ have their values in $TM^\bot$.
We also denote the normal connection of
$TM^\bot$ by $\cd$.

We decompose the equation \eqref{eq:bhmap} into the
tangential direction and the normal direction.
For
$\ol\Dl\tau$, we have
\begin{equation}
\begin{aligned}
&g((\ol\Dl\tau)^\top,\pd_i)
=\ol g((\cd^j\al)(\pd_j,\pd_i),\tau)
+2\ol g(\al(\pd_j,\pd_i),\cd^j\tau),
\\
&(\ol\Dl\tau)^\bot
=\Dl\tau+g^{ij}g^{k\ell}\ol g(\al(\pd_j,\pd_k),\tau)\al(\pd_i,\pd_\ell).
\end{aligned}
\end{equation}
For $g^{jk}\ol R(\tau,\pd_j),\pd_k)^\top$, we have
\begin{equation}
\begin{aligned}
&g((g^{jk}\ol R(\tau,\pd_j),\pd_k))^\top,\pd_i)
=g^{jk}g(\ol R(\pd_i,\pd_k),\pd_j),\tau)\\
&=g^{jk}\ol g((\cd_{\pd_i}\al)(\pd_k,\pd_j)
-(\cd_{\pd_k}\al)(\pd_i,\pd_j),\tau)
\quad(\text{by Codazzi eq.})\\
&=\ol g(\cd_{\pd_i}\tau,\tau)-\ol g((\cd^j\al)(\pd_j,\pd_i),\tau)\\
&=\frac12\cd_i\nr\tau^2-\ol g((\cd^j\al)(\pd_j,\pd_i),\tau).
\end{aligned}
\end{equation}

Now we introduce the following notions:
\begin{equation}
\label{eq:definition-dl,al2,Rbot}
\begin{aligned}
&(\dl\al)(\pd_i):=-(\cd^j\al)(\pd_j,\pd_i),
\\
&(\al^2)(V):=g^{ij}g^{k\ell}\ol g(\al(\pd_j,\pd_k),V)\al(\pd_i,\pd_\ell),
\\
&(\ol R^\bot_V)(\pd_i,\pd_j):=(\ol R(V,\pd_i)\pd_j)^\bot.
\end{aligned}
\end{equation}
By \eqref{eq:bhmap} $\sim$ \eqref{eq:definition-dl,al2,Rbot}, the equations of biharmonic submanifolds can be written as follows:
\begin{lemma}
A submanifold $M$ of $(\ol M,\ol g)$ is a biharmonic submanifold
if and only if the following two equations hold:
\begin{subequations}
\label{eq:bhmfd}
\begin{align}
\label{eq:bhmfd-bot}
&\Dl\tau+\al^2(\tau)-\tr\ol R^\bot_\tau=0,
\\
\label{eq:bhmfd-top}
&-2\ol g((\dl\al)(\pd_i),\tau)+2\ol g(\al(\pd_j,\pd_i),\cd^j\tau)
-\frac12\cd_i\nr\tau^2
=0.
\end{align}
\end{subequations}
\end{lemma}
By using these equations, we first give examples of biharmonic hypersurfaces which are not minimal submanifolds.
In the following, we regard
$\al,\tau$ to be real values in terms of $N$.
We need first the following lemma.
\begin{lemma}
\label{lemma:general-hyper}
In a Riemannian manifold $(\ol M=M\times{\bf R},\,\ol g=g(t)+dt^2)$,
the second fundamental form $\al_{ij}$ and the symmetric bilinear form
$\beta_{ij}:=\ol g(\ol R^\bot_N(\pd_i,\pd_j),N)$ on $M$ are given by
\begin{equation}
\label{eq:general-hyper}
\begin{aligned}
&\al_{ij}
=-\frac12g'_{ij},
\quad
\beta_{ij}
=-\frac12g''_{ij}+\frac14g^{k\ell}g'_{ik}g'_{j\ell},
\end{aligned}
\end{equation}
where, $*'$ means the differentiation with respect to $t$.

As a consequence, for every pair of symmetric bilinear forms
$\al$ and $\beta$ on $M$, there exists a Riemannian metric $\ol g=g(t)+dt^2$ such that $\al$ coincides with
the second fundamental form at $t=0$ and $\beta$ coincides with $\ol g(\ol R^\bot_N(\pd_i,\pd_j),N)$ at $t=0$, respectively.
\end{lemma}
\begin{proof}
We add $t$ to the coordinates $\{x^i\}$ of
$M$ to become the coordinates of $\ol M$.
We denote the differentiation with respect to
$x^i$ by $\pd_i$, and the differentiation with respect to $t$
by $\pd_t$. For the second fundamental form
$\al$, it is well known that
\begin{equation}
\al_{ij}
=\ol g(\ol\cd_{\pd_i}\pd_j,\pd_t)
=-\frac12g'(\pd_i,\pd_j)
=-\frac12g_{ij}'.
\end{equation}

Furthermore, since
$\ol g(\ol\cd_{\pd_i}\pd_t,\pd_t)=0$
and
$
\=\ol g(\ol\cd_{\pd_i}\pd_t,\pd_j)
=-\ol g(\pd_t,\ol\cd_{\pd_i}\pd_j)\allowbreak
=\frac12g_{ij}'
$,
we have
$\=\ol\cd_{\pd_i}\pd_t=\frac12g^{k\ell}g_{ik}'\pd_\ell$.
Thus, since
$
\ol g(\ol\cd_{\pd_t}\pd_t,\pd_i)
=\break-\ol g(\pd_t,\ol\cd_{\pd_t}\pd_i)
=0
$
we have $\ol\cd_{\pd_t}\pd_t=0$.
Therefore, we obtain
\begin{equation}
\begin{aligned}
\ol g(&\ol R(\pd_t,\pd_i)\pd_j,\pd_t)
=\ol g(\ol\cd_{\pd_t}\ol\cd_{\pd_i}\pd_j
-\ol\cd_{\pd_i}\ol\cd_{\pd_t}\pd_j,\pd_t)\\
&=\pd_t\{\ol g(\ol\cd_{\pd_i}\pd_j,\pd_t)\}
	-\ol g(\ol\cd_{\pd_i}\pd_j,\ol\cd_{\pd_t}\pd_t)\\
&\qquad-\pd_i\{\ol g(\ol\cd_{\pd_t}\pd_j,\pd_t)\}
	+\ol g(\ol\cd_{\pd_t}\pd_j,\ol\cd_{\pd_i}\pd_t)\\
&=-\frac12g_{ij}''-0-0+\frac14\ol g(g^{pq}g'_{jp}\pd_q,g^{k\ell}g'_{ki}\pd_\ell)\\
&=-\frac12g''_{ij}+\frac14g^{k\ell}g'_{ki}g'_{\ell j}.
\end{aligned}
\end{equation}

Since for any $\al$ and $\beta$, we can solve \eqref{eq:general-hyper} as a system of equations for $g'$ and $g''$, the latter half statement holds.
\end{proof}

By using Lemma \ref{lemma:general-hyper}, we have the following.
\begin{theorem}
\label{theorem:can-embed}
Every $n$ dimensional Riemannian manifold $M$ can be embedded into an $(n+1)$-dimensional Riemannian manifold $\ol M$ as a biharmonic hypersurface, but not minimal.
\end{theorem}
\begin{proof}
For every $c\ne0$, we apply Lemma \ref{lemma:general-hyper}
to $\al=cg$, $\beta=(c^2/n)g$, and we construct a Riemannian metric $\ol g$. Then, it holds that $\tau=nc$,
which implies \eqref{eq:bhmfd-top},
and that $c^2\cdot nc-nc^2/n\cdot nc=0$,
which implies \eqref{eq:bhmfd-bot}, respectively.
\end{proof}

Note that
our Riemannian metric $\ol g$ satisfies that the sectional curvature
$
(g_{ii})^{-1}\ol g(\ol R(N,\pd_i)\pd_i,N)
=(g_{ii})^{-1}\beta_{ii}
=c^2/n
>0
$.

\section{Biharmonic submanifolds $M^n$ in a space form $\ol M^{n+m}$}
\label{section:BHSM-in-CCS}

For submanifolds $M^n$ in a space form $\overline{M}^{n+m}$ of sectional curvature
$K$, Codazzi equation holds: $\cd_i\al_{jk}=\cd_j\al_{ik}$
(\cite{KN}, Corollary 4.4). Therefore, we have
\begin{equation}
\begin{aligned}
&(\ol R_\tau^\bot)(\pd_i,\pd_j)=Kg(\pd_i,\pd_j)\tau,\quad
\tr\ol R^\bot_\tau=nK\tau,\\
&(\dl\al)(\pd_i)=-(\cd^j\al)(\pd_j,\pd_i)=-(\cd_i\al)(\pd_j,\pd^j)=-\cd_i\tau,\\
&-2\ol g((\dl\al)(\pd_i),\tau)=2\ol g(\cd_i\tau,\tau)=\cd_i\nr\tau^2.
\end{aligned}
\end{equation}
Thus, \eqref{eq:bhmfd} is written as
\begin{subequations}
\begin{align}
\label{eq:bhmfd-R-bot}
&\Dl\tau+\al^2(\tau)-nK\tau=0,
\\
\label{eq:bhmfd-R-top}
&\frac12\cd_i\nr\tau^2+2\ol g(\al(\pd_j,\pd_i),\cd^j\tau)
=0.
\end{align}
\end{subequations}
By taking the inner product of
\eqref{eq:bhmfd-R-bot}
and $\tau$, we have
\begin{equation}
\begin{aligned}
&\ol g(\Dl\tau,\tau)
=-\cd^i(\ol g(\cd_i\tau,\tau))+\ol g(\cd_i\tau,\cd^i\tau)
=\frac12\Dl\nr\tau^2+\nr{\cd\tau}^2,
\\
&\ol g(\al^2(\tau),\tau)
=\nr{\al_\tau}^2,
\end{aligned}
\end{equation}
where we put
$(\al_\tau)_{ij}:=\ol g(\al(\pd_i,\pd_j),\tau)$.
Therefore, $\nr\tau^2$ satisfies
\begin{equation}
\frac12\Dl\nr\tau^2+\nr{\cd\tau}^2+\nr{\al_\tau}^2-nK\nr\tau^2=0.
\end{equation}
From this elliptic equation for $\nr\tau^2$, we have the following

\begin{proposition}[\cite{On} Proposition 2.4]
\label{proposition:tau-is-nonconst}
Assume that $\ol M^{n+m}$ is a space form of
constant sectional curvature $K$ less than or equal to $0$,
and that $M$ is a biharmonic submanifold of $\ol M$.
If $\nr\tau^2$ admits a local maximum at some point, then $M^n$ is minimal.
In particular, if $\nr\tau^2$ is constant, then $M$ is minimal.

\end{proposition}
\begin{proof}
If $\nr\tau^2$ admits a local maximum at some point,
$\Dl\nr\tau^2\ge0$ holds at the point.
Therefore, we have
$\nr{\al_\tau}^2=0$, which implies that $\tau=0$. Thus, the local maximum must be $0$, so it hold that $\tau\equiv0$, locally.
By real analyticity of biharmonic submanifolds, $\tau$ must
be $0$.
\end{proof}

\section{An over-determined system of ODE}
\label{section:BH-hypersurface-in-CCS}

In this section, we assume that the ambient space $\ol M^{n+1}$ is a space form of sectional curvature $K$, and $M^n$ is a biharmonic hypersurface in $\ol M$.

We can study more precisely
biharmonic hypersurfaces
since we can diagonalize the second fundamental form.
Indeed, by taking an orthonormal frame consisting of unit principal curvature vectors
$\{v_i\}$,
we can diagonalize the second fundamental form as
$\al(v_i,v_j)=\lm_i\dl_{ij}$.
Then, by
\eqref{eq:bhmfd-R-top}, we have
\begin{subequations}
\begin{align}
\label{eq:base-n}
&\Dl\tau+\nr\al^2\tau-nK\tau=0,
\\
\label{eq:base-t}
&\tau\cd_i\tau+2\al_{i\ell}\cd^\ell\tau
=(\tau+2\lm_i)\cd_i\tau=0\quad(\text{for all $1\le i\le n$}).
\end{align}
\end{subequations}

Therefore, in the case of $K\le0$, if we assume that
there is no principal curvature satisfying that
$\tau+2\lm_i=0$,
then $\tau$ must be constant, and then $M^n$ is minimal.
Thus, we have

\begin{proposition}
\label{proposition:tau+lm-n=0}
If $M^n$ is a non-minimal biharmonic submanifold of the space form $\ol M^{n+1}$ of constant curvature $K$ with $K\le0$,
$(-1/2)\tau$ is a principal curvature.
\end{proposition}

Proposition \ref{proposition:tau+lm-n=0} is essentially important
to continue our arguments below to obtain main theorem.

\medskip
From now on, we assume that $M$ is
a biharmonic hypersurface, but not minimal.

Thus, in the case that
$K\le0$, the mean curvature is not constant.
In the case that $K\ge0$, we assume that the mean curvature is not constant.
We always assume that $n\ge2$.

Let
$\{\lm_i\}$ be the principal curvatures, and let us denote their unit principal curvature vectors by
$\{v_i\}$,
and put
$\tau:=\ssum\lm_i$.
In the following, all the subscripts of the tensor fields mean the ones with respect to not the local coordinates, but $\{v_i\}$.
For examples,
$\al_{ij}=\al(v_i,v_j)=\dl_{ij}\lm_i$.
And we denote the differentiation with respect to $v_i$ by $v_i[*]$.
Furthermore,
$
g(\cd_{v_i}v_j,v_k)
=v_i[g(v_j,v_k)]-g(v_j,\cd_{v_i}v_k)
=-g(v_j,\cd_{v_i}v_k)
$,
in particular, $g(\cd_{v_i}v_j,v_j)=0$,
which we will use frequently.

Note that $v_i$ is not uniquely determined when $\lm_i$ has multiplicity.
To select them suitably, we need the following formula.

\begin{lemma}
For all $i,j,k\le n$, we have
\begin{subequations}
\begin{align}
\label{eq:Codazzi-cd-al}
&\cd_i\al_{jk}=\dl_{jk}v_i[\lm_j]+(\lm_j-\lm_k)g(\cd_{v_i}v_j,v_k),\\
\label{eq:Codazzi-1}
&\begin{aligned}
\dl_{jk}v_i[\lm_j]&-\dl_{ik}v_j[\lm_i]\\
&=-(\lm_j-\lm_k)g(\cd_{v_i}v_j,v_k)+(\lm_i-\lm_k)g(\cd_{v_j}v_i,v_k).
\end{aligned}
\end{align}
\end{subequations}
\end{lemma}

\begin{proof}
Using $\al_{ij}=\dl_{ij}\lm_i$,
\begin{equation}
\begin{aligned}
\cd_i\al_{jk}&=v_i[\al_{jk}]-\al(\cd_{v_i}v_j,v_k)-\al(v_j,\cd_{v_i}v_k)\\
&=v_i[\dl_{jk}\lm_j]-\lm_kg(\cd_{v_i}v_j,v_k)-\lm_jg(v_j,\cd_{v_i}v_k)\\
&=\dl_{jk}v_i[\lm_j]-\lm_kg(\cd_{v_i}v_j,v_k)+\lm_jg(\cd_{v_i}v_j,v_k).
\end{aligned}
\end{equation}
Hence the first equation holds.
The second equation is derived from Codazzi equation: $\cd_i\al_{jk}=\cd_j\al_{ik}$.
\end{proof}

In case $k=j\ne i$ and $\lm_j=\lm_i$ in \eqref{eq:Codazzi-1}, we have $v_i[\lm_j]=0$.
It means that if $\lm_i$ has multiplicity $>1$, then $v_i[\lm_i]=0$.
We renumber the indices $\{i\}$ so that $\tau+2\lm_n=0$.
From \eqref{eq:base-t}, if $\lm_i\ne\lm_n$, then $v_i[\lm_n]=0$.
Therefore, if $\lm_n$ has multiplicity $>1$, then $v_i[\lm_n]=0$ for all $i\le n$, $\tau$ is constant on $M$, and $M$ is minimal.
Since $M$ is not minimal, we conclude that $\lm_n$ is simple.

For other $\lm_i$ with multiplicity $>1$, we reselect $v_i$ as follows.
For an index $i_0$, let $E$ be the tangent sub-bundle on $M$ generated by $\{v_i\mid\lm_i=\lm_{i_0}\}$.
Since $\lm_n$ is simple, $E$ has trivial normal connection $\cd^{E\bot}$ along each integral $v_n$ curve.
Therefore, we can choose orthonormal bases $\{v_i\}$ of $E$ so that $\cd^{E\bot}_{v_n}v_i=0$, i.e., $g(\cd_{v_n}v_i,v_j)=0$ for all $i$ and $j$ satisfying $\lm_i=\lm_j=\lm_{i_0}$.
Moreover, we have $v_i[\tau]=-2v_i[\lm_n]=0$ for $i<n$.

We summarize the above selection of $\{v_i\}$ as follows,
and, from now on, we assume the frame field $\{v_i\}$ satisfies the property.

\begin{lemma}
\label{lemma:selection-vi}
The principal curvature $\lm_n=-\tau/2$ is simple,
and it holds that $v_i[\tau]=0$ for any $i<n$.
Moreover, we can choose the frame field $\{v_i\}$ so that
$g(\cd_{v_n}v_i,v_j)=0$ if $\lm_i=\lm_j$.
\end{lemma}

Moreover, the chosen vector fields $v_i$ have the following good property.

\begin{lemma}
\label{lemma:cdvnvi=0}
For any $i\le n$, it holds that
\begin{equation}
\cd_{v_n}v_i=0.
\end{equation}
\end{lemma}

\begin{proof}
We consider
the covariant differentiation of
\eqref{eq:base-t}.
\begin{equation}
\label{eq:base-t-d}
\begin{aligned}
0&=\cd_j(\tau\cd_i\tau+2\al_{i\ell}\cd^\ell\tau)\\
&=\cd_j\tau\cd_i\tau+\tau\cd_j\cd_i\tau
+2\cd_j\al_{i\ell}\cd^\ell\tau+2\al_{i\ell}\cd_j\cd^\ell\tau\\
&=\cd_j\tau\cd_i\tau+\tau\cd_j\cd_i\tau
+2\cd_j\al_{in}\cd_n\tau+2\lm_i\cd_j\cd_i\tau\\
&=(\tau+2\lm_i)\cd_i\cd_j\tau+2\cd_n\tau\cd_n\al_{ij}+\cd_i\tau\cd_j\tau.
\end{aligned}
\end{equation}

Exchanging $i$ and $j$ in \eqref{eq:base-t-d}, and taking the difference between them,
we have $(\lm_i-\lm_j)\cd_i\cd_j\tau=0$.
Assume that $\lm_i\ne\lm_j$.
Then it holds that $\cd_i\cd_j\tau=0$.
Substituting it into \eqref{eq:base-t-d}, we have $\cd_n\al_{ij}=0$,
because $\cd_i\tau=0$ or $\cd_j\tau=0$ and $\cd_n\tau\ne0$.
Substitute it into \eqref{eq:Codazzi-cd-al} replaced $i,j,k$ by $n,i,j$.
Then, we have $0=(\lm_i-\lm_j)g(\cd_{v_n}v_i,v_j)$.
Therefore, $g(\cd_{v_n}v_i,v_j)=0$ if $\lm_i\ne\lm_j$.

On the other hand, $g(\cd_{v_n}v_i,v_j)=0$ if $\lm_i=\lm_j$ by Lemma \ref{lemma:selection-vi}.
Thus, we have $g(\cd_{v_n}v_i,v_j)=0$ for any $i,j\le n$.
We also have
\begin{equation}
\cd_n\al_{ij}=0
\end{equation}
for $i,j\le n$ and $i\ne j$, by \eqref{eq:Codazzi-cd-al}.
\end{proof}

Since $\tau$ is not constant,
on a neighborhood of a point satisfying that $d\tau\ne0$,
the set $\tau=(\text{a constant})$ is a hypersurface of $M$.

\begin{definition}
We call each $F$ defined by $\tau=(\text{a constant})$ {\em a characteristic hypersurface} of $M$.
\end{definition}

For $i<n$, \eqref{eq:base-t} implies that $v_i[\tau]=0$, because $\lm_i\ne\lm_n=-\tau/2$.
Therefore, the set $\{v_i\mid1\le i<n\}$ is a locally defined orthonormal frame field of the tangent bundle of $F$.

Moreover, every $v_i$ is a principal curvature vector filed of $F$ as follows.

\begin{lemma}
\label{lemma:beta-mu}
Every vector $v_i$ $(i<n)$ is a principal curvature vector of hypersurface $F$ in $M$.
We denote by $\beta$ the second fundamental form, and by $\mu_i$ the principal curvature for the direction $v_i$.
Then, it holds that, for $i,j<n$,
\begin{equation}
\label{eq:beta-diag}
\beta(v_i,v_j)=g(\cd_{v_i}v_j,v_n)=\dl_{ij}\mu_i,
\quad\cd_{v_i}v_n=-\mu_iv_i.
\end{equation}
\end{lemma}

\begin{proof}
We consider \eqref{eq:Codazzi-1} with $i=n$, $j,k<n$ and $j\ne k$.
We have
\begin{equation}
0=-(\lm_j-\lm_k)g(\cd_{v_n}v_j,v_k)+(\lm_n-\lm_k)g(\cd_{v_j}v_n,v_k).
\end{equation}
Since $\cd_{v_n}v_j=0$ by Lemma \ref{lemma:cdvnvi=0} and $\lm_k\ne\lm_n$,
we have
\begin{equation}
0=g(\cd_{v_j}v_n,v_k)=-g(v_n,\cd_{v_j}v_k)=-\beta(v_j,v_k).
\end{equation}

It also implies that $\cd_{v_j}v_n$ is parallel to $v_j$.
Therefore, $\cd_{v_j}v_n=g(\cd_{v_j}v_n,v_j)v_j=-\mu_jv_j$.
\end{proof}

Now, we derive an over-determined ODE.

\begin{proposition}
\label{proposition:ODE-0}
Let $M^n$ be a biharmonic hypersurface of the space form $\ol M^{n+1}$ of sectional curvature $K\le0$.
Then, $\lm_i$ and $\mu_i$ satisfy the following ordinary differential equations:
\begin{subequations}
\label{eq:ODE-0}
\begin{align}
\label{eq:ODEtau}
&-\tau''+\tau'\ssum_{i<n}\mu_i
+\tau(\frac14\tau^2-nK+\ssum_{i<n}\lm_i^2)
=0,\\
\label{eq:ODElm}
&(\lm_i)'=(\frac12\tau+\lm_i)\mu_i,\\
\label{eq:ODEmu}
&(\mu_i)'=\mu_i^2-\frac12\tau\lm_i+K.
\end{align}
\end{subequations}
Here, $\tau$ is a function of $\{\lm_i\}$ defined by
$\tau=(2/3)\ssum_{i<n}\lm_i$,
and $*'$ is the differentiation $v_n[*]$.
\end{proposition}

\begin{proof}
For \eqref{eq:ODEtau}, we calculate $\cd_i\cd_j\tau$ for $i,j\le n$.
\begin{equation}
\begin{aligned}
\cd_i\cd_j\tau&=(\cd_{v_i}(\cd\tau))(v_j)
=v_i[v_j[\tau]]-\cd_{\cd_{v_i}v_j}\tau\\
&=v_i[v_j[\tau]]-\ssum_{k\le n}g(\cd_{v_i}v_j,v_k)\cd_{v_k}\tau\\
&=v_i[v_j[\tau]]-g(\cd_{v_i}v_j,v_n)\cd_n\tau.
\end{aligned}
\end{equation}
Therefore, $\cd_i\cd_i\tau=-\mu_i\tau'$ for all $i<n$, and $\cd_n\cd_n\tau=\tau''$.
Substituting it into
\eqref{eq:base-n}, we have an expression of $\Dl\tau$:
\begin{equation}
\Dl\tau=-\cd_n\cd_n\tau-\ssum_{i<n}\cd_i\cd_i\tau
=-\tau''+\tau'\ssum_{i<n}\mu_i.
\end{equation}
Thus, we have
\begin{equation}
\begin{aligned}
0&=\Dl\tau+\nr\al^2\tau-nK\tau\\
&=-\tau''+\tau'\ssum_{i<n}\mu_i
+\tau(\frac14\tau^2-nK+\ssum_{i<n}(\lm_i)^2),
\end{aligned}
\end{equation}
which is \eqref{eq:ODEtau}.

For $v_n[\lm_i]$, we use \eqref{eq:Codazzi-1} with $j=n$, $k=i<n$.
We have, using \eqref{lemma:beta-mu},
\begin{equation}
\begin{aligned}
&-v_n[\lm_i]=-(\lm_n-\lm_i)g(\cd_{v_i}v_n,v_i)
=(-\frac12\tau-\lm_i)\mu_i,
\end{aligned}
\end{equation}
which is \eqref{eq:ODElm}

For $v_n[\mu_i]$, differentiating the equation of the definition of $\mu_i$, we have
\begin{equation}
\begin{aligned}
v_n&[\mu_i]=v_n[g(\cd_{v_i}v_i,v_n)]
=g(\cd_{v_n}\cd_{v_i}v_i,v_n)+g(\cd_{v_i}v_i,\cd_{v_n}v_n)\\
&=g(R(v_n,v_i)v_i+\cd_{v_i}\cd_{v_n}v_i
+\cd_{[v_n,v_i]}v_i,v_n)
\,\,(\text{by }\cd_{v_n}v_n=0)\\
&=K+\lm_n\lm_i+g(\cd_{\cd_{v_n}v_i-\cd_{v_i}v_n}v_i,v_n)
\,\,(\text{by Gauss, }\cd_{v_n}v_i=0)\\
&=K-\frac12\tau\lm_i+\mu_ig(\cd_{v_i}v_i,v_n)
\,(\text{by}\,\cd_{v_n}v_i=0,\!\cd_{v_i}v_n=-\mu_iv_i)\\
&=K+\mu_i^2-\frac12\tau\lm_i,
\end{aligned}
\end{equation}
which is \eqref{eq:ODEmu}.
\end{proof}

In the case $K\le0$, note here that we can conclude non-existence of non-minimal biharmonic hypersurfaces if the overdetermined differential system \eqref{eq:ODE-0} has only solutions satisfying $\tau\equiv\text{const}$.

In the case $n=2$, we can prove by a different manner the B.-Y. Chen's theorem:

\begin{corollary} [Chen \cite{C1}, Jiang \cite{J})]
Every biharmonic submanifold in
the $3$ dimensional space form of non-positive sectional curvature is minimal.
\end{corollary}

\begin{proof}
Substituting $n=2$,
\eqref{eq:ODElm}, \eqref{eq:ODEmu} and $\tau'=(2/3)\lm_1'$ into \eqref{eq:ODEtau}, we have
\begin{equation}
(2/27)\lm_1\{14(\lm_1)^2-16(\mu_1)^2-9nK-12K\}=0.
\end{equation}
Since $\lm_1=(3/2)\tau$ is not constant,
\begin{equation}
\label{eq:n1=1,2nd}
14(\lm_1)^2-16(\mu_1)^2-9nK-12K=0.
\end{equation}
Substituting \eqref{eq:ODElm} and \eqref{eq:ODEmu} into
the equation
$\{14(\lm_1)^2-16(\mu_1)^2-9nK-12K\}'=0$, we have
\begin{equation}
-16\mu_1\{3(\lm_1)^2-2(\mu_1)^2-2K\}=0.
\end{equation}
If $\mu_1\equiv0$, by \eqref{eq:ODElm}, $\tau=(2/3)\lm_1$ must be constant. Thus, the submanifold is minimal.
If $\mu_1\ne0$ at a point, we have that
$3(\lm_1)^2-2(\mu_1)^2-2K=0$.
Then, together with \eqref{eq:n1=1,2nd}, $\lm_1$ is constant. We have done.
\end{proof}

\section{Solutions to the over-determined system of ODE}
\label{section:analysis-overdetermined-ODE}

In this section, we assume that $M^n$ is a biharmonic hypersurface in the space form $\ol M^{n+1}$ of sectional curvature $K$.

We will analyze the algebraic ordinary differential system \eqref{eq:ODE-0}.
In the following we only treat the case $i\le n-1$, and put $n_1:=n-1$.
And for the summations which will be treated in this section in
$i$, we always assume that $i$ run over the set $1\le i\le n_1$.
In this section, we will treat a more general setting that
it would occur that
$\lm_i=\lm_j$ for different $i$ and $j$.

We will denote by $\lm=(\lm_1,\cdots,\lm_{n_1})\in{\bf R}^{n_1}$,
and $\mu=(\mu_1,\cdots,\mu_{n_1})\in{\bf R}^{n_1}$.
The solutions to the ordinary differential system can be regarded
as real analytic maps on a neighborhood of the origin of
${\bf R}$ into ${\bf R}^{2n_1}=\{(\lm,\mu)\}$,
and
our $\tau$ is regarded as a function $\tau=(2/3)\ssum_{i\le n_1}\lm_i$ on ${\bf R}^{2n_1}$.

Let
$S\subset{\bf R}^{2n_1}$ be the set of all the $(\lambda,\mu)$
such that
the equation \eqref{eq:ODE-0} has a solution with
initial value $(\lambda,\mu)$.

\begin{lemma}
\label{lemma:P0}
Each $(\lm,\mu)\in S$ is a zero point of the following polynomial.
\begin{equation}
\label{eq:P0}
\begin{aligned}
P_0&:=-\frac43\ssum\lm_i(\mu_i)^2
-\frac23\tau\ssum(\mu_i)^2
+\frac49\ssum\mu_i\ssum{}\lm_i\mu_i
+\frac43\tau\ssum(\lm_i)^2\\
&\qquad
+\frac29\tau(\ssum\mu_i)^2
+\frac12\tau^3
-\frac23(2n+1)K\tau.
\end{aligned}
\end{equation}
\end{lemma}

\begin{proof}
Assume that
$(\lm,\mu)=(\lm(t),\mu(t))$ is a solution to \eqref{eq:ODE-0}.
Then,
\begin{equation}
\begin{aligned}
\tau'&=\frac23\ssum\lm_i'
=\frac23\ssum(\frac12\tau\mu_i+\lm_i\mu_i)
=\frac13\tau\ssum\mu_i+\frac23\ssum{}\lm_i\mu_i,\\
\tau''&=\frac13\tau'\ssum\mu_i+\frac13\tau\ssum{}\mu_i'
+\frac23\ssum\lm_i'\mu_i+\frac23\ssum\lm_i\mu_i'.
\end{aligned}
\end{equation}
Substituting this, \eqref{eq:ODElm} and \eqref{eq:ODEmu} into \eqref{eq:ODEtau},
we obtain the desired polynomial.
\end{proof}

Starting at $P_0$,
we determine the polynomial $P_k$ inductively as follows: Substituting the solution $(\lm(t),\mu(t))$ to
the ordinary differential equation of normal form,
\eqref{eq:ODElm}, \eqref{eq:ODEmu}
in the polynomial $P_k$, and differentiate it with respect to
$t$,
and substitute \eqref{eq:ODElm}, \eqref{eq:ODEmu} into $\lm_i'(t),\mu_i'(t)$.
Then, we obtain a polynomial in $\{\lm_i(t),\mu_i(t)\}$.
We define $P_{k+1}$, this polynomial.

\begin{proposition}
The set $S$ of all initial values of \eqref{eq:ODE-0} coincides with the algebraic manifold $\cap_{k=0}^\infty(P_k)^{-1}(0)$.
\end{proposition}

\begin{proof}
We proved that $S\subset\cap_{k=0}^\infty(P_k)^{-1}(0)$.
Conversely, let $(\lm,\mu)$ be a point of $\cap_{k=0}^\infty(P_k)^{-1}(0)$, and $(\lm(t),\mu(t))$ be the solution to the partial system \eqref{eq:ODElm}, \eqref{eq:ODEmu} with initial value $(\lm,\mu)$.
Then, the $k$-th derivative of the left hand side of \eqref{eq:ODEtau} vanishes for any $k\ge0$.
Since the solution $(\lm(t),\mu(t))$ is real analytic, it means that \eqref{eq:ODEtau} is satisfied.
\end{proof}

For the ODE \eqref{eq:ODE-0}, we can classify all solutions giving minimal hypersurfaces.

\begin{proposition}
All the solutions to \eqref{eq:ODE-0} satisfying that
$\tau\equiv0$ are classified as follows:

In the case that $K=0:$
\begin{equation}
\label{eq:lm-mu:K=0}
\begin{aligned}
\mu_i&=-\frac{1}{t+c_i},\;
\lm_i=\frac{a_i}{t+c_i},
\quad\text{or}\quad
\mu_i=0,\;
\lm_i=a_i.
\end{aligned}
\end{equation}

In the case that $K=-1:$
\begin{equation}
\label{eq:lm-mu:K=-1}
\begin{aligned}
&\mu_i=-\tanh(t+c_i),\;
\lm_i=\frac{a_i}{\cosh(t+c_i)},\\
&\text{or}\quad
\mu_i=\pm1,\;
\lm_i=a_ie^{\pm t}.
\end{aligned}
\end{equation}

In the case that $K=1:$
\begin{equation}
\label{eq:lm-mu:K=1}
\begin{aligned}
\mu_i&=\tan(t+c_i),\;
\lm_i=\frac{a_i}{\cos(t+c_i)}.
\end{aligned}
\end{equation}

Here,
the sum of all the $a_i$
corresponding to the same $c_i$ must be zero:
$\ssum_{\{i\,\mid\;c_i=c_k\}}a_i=0$ for every $k$.
(regarding $c_i=\infty$ for the singular solution in the case
$K=0$, and $c_i=\pm\infty$ in the case that
$K=-1$).
For example, if all the $c_i$ are different each other,
it must be that $a_i=0$ for all $i$.
If all $c_i$ are same, our condition is only that $\ssum_ia_i=0$.
\end{proposition}

\begin{proof}
Substituting $\tau=0$ in system
\eqref{eq:ODE-0},
equation \eqref{eq:ODEtau} is automatically satisfied, and we have
\begin{equation}
\lm_i'=\lm_i\mu_i,\quad
\mu_i'=(\mu_i)^2+K.
\end{equation}
We solve the second equation, and substitute the solution into the first equation.
Then we get \eqref{eq:lm-mu:K=0}, \eqref{eq:lm-mu:K=-1}, or \eqref{eq:lm-mu:K=1}.
Each solution satisfies $\ssum\lm_i=(3/2)\tau=0$ if and only if the last condition for $a_i$ is satisfied.
\end{proof}

\begin{proposition}
There exist constant solutions to
\eqref{eq:ODE-0} which satisfy $\tau\ne0$ only in the case $n=4$, $K>0$, and it holds that
$\lm_i=\pm\sqrt{K}$, $\mu_i=0$, $\tau=\pm2\sqrt{K}$ in this case.
\end{proposition}

\begin{proof}
Substitute $\lm_i'=0$, $\mu_i'=0$, $\tau'=\tau''=0$ into
\eqref{eq:ODE-0}, we have
\begin{equation}
\label{eq:const-solutions1}
\begin{aligned}
&(\frac12\tau+\lm_i)\mu_i=0,\quad
(\mu_i)^2-\frac12\tau\lm_i+K=0,\\
&\frac14\tau^2-nK+\ssum_i(\lm_i)^2=0.
\end{aligned}
\end{equation}
By the third one of \eqref{eq:const-solutions1}, we have that $K>0$,
because $\tau\ne0$.
In the case that $\lm_i=-\tau/2$ in the first and second ones of \eqref{eq:const-solutions1},
we have $(\mu_i)^2+(1/4)\tau^2+K=0$
which does not occur.
Thus, we have $\mu_i=0$.
Then, we have by the second one,
$\lm_i=2K/\tau$.
By using together with $\tau=(2/3)n_1\lm_i$, we have
$\tau=\pm(2/\sqrt3)\sqrt{n_1K}$, $\lm_i=\pm\sqrt{3K}/\sqrt{n_1}$.
Substituting this into the third one of \eqref{eq:const-solutions1},
we obtain that
$n=4$ and the other claims.
\end{proof}

\begin{corollary}
In the case of the space form of constant curvature,
the biharmonic hypersurface all of whose principal curvatures are constant and different each other must be minimal.
\end{corollary}

\begin{remark}
There exist examples of
biharmonic hypersurfaces having principal curvatures with multiplicities, and they are classified under their
completeness conditions
(Ichiyama, Inoguchi and Urakawa,
\cite{IIU1}, \cite{IIU2}).
They are only the case $K>0$.
In the case that $K\le0$,
there are no such biharmonic hypersurfaces which are not minimal by Proposition \ref{proposition:tau-is-nonconst}.
\end{remark}

\begin{lemma}
\label{lemma:same-lm<=>same-mu}
Let $\{\lm_i,\mu_i\}$ be a solution to \eqref{eq:ODE-0} with $\tau\ne0$.
If $\lm_i\equiv\lm_j$, then $\mu_i\equiv\mu_j$ or $\lm_i\equiv-(1/2)\tau$.
Conversely, if $\mu_i\equiv\mu_j$, then $\lm_i\equiv\lm_j$.
\end{lemma}

\begin{proof}
If $\lm_i\equiv\lm_j$, then $(\tau/2+\lm_i)(\mu_i-\mu_j)\equiv0$,
hence $\tau/2+\lm_i\equiv0$ or $\mu_i\equiv\mu_j$.
Conversely, if $\mu_i\equiv\mu_j$, then $\tau(\lm_i-\lm_j)\equiv0$.
Since $\tau\ne0$, we have $\lm_i\equiv\lm_j$.
\end{proof}

Note that the case $\lm_i\equiv-\tau/2$ does not occur when the solution comes from a non-minimal biharmonic hypersurface, because $\lm_n=-\tau/2$ is simple.
Next we consider solutions with same $\lm_i$, under $K=0$.

\begin{lemma}
\label{lemma:same-lm}
The solution to
\eqref{eq:ODE-0} with $K=0$ satisfying
that
all the $\lm_i$ are the same, must satisfy
$\tau\equiv0$.
\end{lemma}

\begin{proof}
We assume $\tau\ne0$.
We may write as $\lm_i=\lm$.
Since $\tau=(2/3)n_1\lm$, we get $\lm\ne-(1/2)\tau$.
Thus, by Lemma \ref{lemma:same-lm<=>same-mu}, all the $\mu_i$ are equal to each other, we may write as $\mu_i=\mu$.
Substituting these into \eqref{eq:ODE-0}, we have
\begin{equation}
\lm'=\frac{n_1+3}3\lm\mu,
\,\,
\mu'=\mu^2-\frac{n_1}3\lm^2,
\,\,
\frac{n_1(n_1+9)}9\lm^3+n_1\mu\lm'-\lm''=0.
\end{equation}
Differentiating the first equation, and substituting into the third one into which the first and second equations are substituted,
we obtain
\begin{equation}
\lm\cdot\{n_1(n_1+6)\lm^2+\bigl((n_1)^2-9\bigr)\mu^2\}=0.
\end{equation}
Thus, $\tau=0$ if $n_1\ge3$.

In the case that $n_1\le2$, we have
$\mu=c\,\lm$.
Substituting this into
$\lm'$, $\mu'$, and eliminating $\lm'$, we have
$(1+c^2)\lm^2=0$ which implies that $\tau=0$.
\end{proof}

We can also solve \eqref{eq:ODElm}, \eqref{eq:ODEmu} with $K=0$ if we know the function $\tau$.

\begin{proposition}
\label{proposition:r-theta}
The ordinary differential system \eqref{eq:ODElm}, \eqref{eq:ODEmu} with $K=0$ can be solved
as follows if we regard $\tau$ as a known function.
If we put $\lm_i=r_i\sin\te_i$, $\mu_i=r_i\cos\te_i$,
\begin{equation}
\te_i=\frac12\int\tau\,dt,\quad
r_i=\frac{-1}{\int\cos\te_i\,dt}.
\end{equation}
\end{proposition}

\begin{proof}
Let us rewrite
\eqref{eq:ODElm}, \eqref{eq:ODEmu} in terms of $r_i$, $\te_i$,
\begin{equation}
\label{eq:d(r-te)-(r-te)}
\begin{aligned}
&r_i'\sin\te_i+\te_i'r_i\cos\te_i=(\frac12\tau+r_i\sin\te_i)r_i\cos\te_i,\\
&r_i'\cos\te_i-\te_i'r_i\sin\te_i=(r_i)^2\cos^2\te_i-\frac12\tau\,r_i\sin\te_i.
\end{aligned}
\end{equation}
By \eqref{eq:d(r-te)-(r-te)}, we have $r_i'=(r_i)^2\cos\te_i$, and $\te_i'r_i=(1/2)\tau r_i$ which solve $\te_i$, and $r_i$.
\end{proof}

\section{Irreducibility of principal curvature vector fields}
\label{section:irreducibility}

In this section, we assume that $M^n$ is a non-minimal biharmonic hypersurface in the Euclidean space $E^{n+1}$, and all the principal curvatures $\{\lm_i\}$ are simple.

We denote the covariant differentiation on the characteristic hypersurface
$F$ by $\wt\cd$, and
the second fundamental form of $F$ in $\ol M$ by $\wt\al$.
The second fundamental form of $F$ in $M$ is denoted by $\beta$.

Note that the unit normal frame fields $N$ and $v_n$ on
$M$ are parallel with respect to the normal connection on $F$ in $E^{n+1}$.
In fact, since
$\ol g(\ol\cd_{v_i}N,v_n)=-\ol g(N,\ol\cd_{v_i}v_n)=-\al_{in}=0$
for every $i<n$, we have
$\ol g(\wt\cd_{v_i}N,v_n)=0$, and $\wt\cd_{v_i}N=0$, $\wt\cd_{v_i}v_n=0$.

Since the $N$-component of $\wt\al$ coincides with
the restriction of $\al$ to the tangent space $TF$ of $F$ because of
$\ol g(\ol\cd_{v_i}v_j,N)=\al(v_i,v_j)$,
we use the same notation $\al$ for it.
The $v_n$-component $\beta$ of
$\wt\al$ is
$\beta_{ij}=\dl_{ij}\mu_i$ by \eqref{eq:beta-diag}.
Thus,
$\wt\al$ can be diagonalized by the frames $\{N,v_n\}$ whose eigenvalues are $\lm_i$ and $\mu_i$.

We calculate the covariant differentiation of
$\wt\al$ as follows:
\begin{equation}
\begin{aligned}
\wt\cd_i\wt\al_{jk}
&=(\wt\cd_{v_i}\wt\al)(v_j,v_k)\\
&=\wt\cd_{v_i}(\wt\al(v_j,v_k))
-\wt\al(\wt\cd_{v_i}v_j,v_k)-\wt\al(v_j,\wt\cd_{v_i}v_k)\\
&=\wt\cd_{v_i}(\dl_{jk}(\lm_kN+\mu_kv_n))
-g(\wt\cd_{v_i}v_j,v_k)(\lm_kN+\mu_kv_n)\\
&\qquad-g(v_j,\wt\cd_{v_i}v_k)(\lm_jN+\mu_jv_n)\\
&=\dl_{jk}(v_i[\lm_k]N+v_i[\mu_k]v_n)
-g(\wt\cd_{v_i}v_j,v_k)(\lm_kN+\mu_kv_n)\\
&\qquad+g(\wt\cd_{v_i}v_j,v_k)(\lm_jN+\mu_jv_n)\\
&=\{\dl_{jk}v_i[\lm_k]+(\lm_j-\lm_k)g(\wt\cd_{v_i}v_j,v_k)\}N\\
&\qquad
+\{\dl_{jk}v_i[\mu_k]+(\mu_j-\mu_k)g(\wt\cd_{v_i}v_j,v_k)\}v_n.
\end{aligned}
\end{equation}

Since
$\wt\cd_i\wt\al_{jk}=\wt\cd_j\wt\al_{ik}$, the $N$-component of $\wt\al$ coincides with
\begin{equation}
\dl_{jk}v_i[\lm_k]+(\lm_j-\lm_k)g(\wt\cd_{v_i}v_j,v_k)
=\dl_{ik}v_j[\lm_k]+(\lm_i-\lm_k)g(\wt\cd_{v_j}v_i,v_k).
\end{equation}
For $i\ne j=k$, it holds that
\begin{equation}
\label{eq:v_i[lm_j]}
v_i[\lm_j]
=(\lm_i-\lm_j)g(\wt\cd_{v_j}v_i,v_j).
\end{equation}
In the case that all the $i,j$ and $k$ are different each other,
\begin{equation}
\wt\cd_i\al_{jk}=(\lm_j-\lm_k)g(\wt\cd_{v_i}v_j,v_k)
=(\lm_i-\lm_k)g(\wt\cd_{v_j}v_i,v_k).
\end{equation}
We conclude that the quantities $s^\lm_{ijk}$ defined by
\begin{equation}
\label{eq:s^lm}
s^\lm_{ijk}:=(\lm_j-\lm_k)g(\wt\cd_{v_i}v_j,v_k)
\end{equation}
are symmetric for all distinct triplets $\{i,j,k\}$.

By the same way, we obtain the relations of $\mu$,
by considering the
$v_n$-component:
For all the $i,j$ and $k$ which are different each other, we can conclude that
\begin{equation}
\label{eq:v_i[mu_j]}
\begin{aligned}
&v_i[\mu_j]=(\mu_i-\mu_j)g(\wt\cd_{v_j}v_i,v_j),\\
&s^\mu_{ijk}:=(\mu_j-\mu_k)g(\wt\cd_{v_i}v_j,v_k)
\text{ are symmetric for all the $i,j,k$}.
\end{aligned}
\end{equation}

Assume that all the principal curvatures are simple.
If $g(\wt\cd_{v_i}v_j,v_k)\ne0$, it holds that $s^\lm_{ijk}\ne0$,
which implies that
$g(\wt\cd_{v_j}v_k,v_i)$, $g(\wt\cd_{v_k}v_i,v_j)\ne0$.
Thus, we obtain the relations that
\begin{equation}
\frac{\mu_i-\mu_j}{\lm_i-\lm_j}
=\frac{\mu_j-\mu_k}{\lm_j-\lm_k}
=\frac{\mu_k-\mu_i}{\lm_k-\lm_i}
=\frac{s^\mu_{ijk}}{s^\lm_{ijk}}.
\end{equation}
Therefore, if
$g(\wt\cd_{v_i}v_j,v_k)\ne0$ for every distinct triplet $\{i,j,k\}$,
then
all the
$\=\frac{\mu_i-\mu_j}{\lm_i-\lm_j}$
coincide each other for every distinct pair $\{i,j\}$.
Thus, if we denote the common quantity by $\vp$, then
all $\mu_i-\vp\lm_i$ have the same value.
If we denote it by $\psi$, then
it holds that $\mu_i-\vp\lm_i=\psi$ for all $i$.

Conversely, if there exist $\vp$ and $\psi$
satisfying that $\mu_i=\vp\lm_i+\psi$ for all $i$, and
there exist at least two different $\lm_i$, $\vp$ and $\psi$ are uniquely determined.
Really, we assume the following weaker conditions:

\begin{definition}
\label{definition:irreducible}
Put $J=\{\{i,j\}\mid 1\le i,j\le n_1,\; i\ne j\}$.
If a distinct triplet $\{i,j,k\}$ satisfies $g(\cd_{v_i}v_j,v_k)\ne0$, then we define $\{i,j\}\sim\{j,k\}\sim\{i,k\}$.
Let $\sim_J$ be the equivalence relation on $J$ generated by $\sim$.
If all $\{i,j\}\in J$ are equivalent under $\sim_J$, the frame field $\{v_i\}$ is {\em irreducible}.
Otherwise, the frame field is {\em reducible}.
\end{definition}

\begin{definition}
If there exist functions $\vp$, $\psi$ satisfying $\mu_i=\vp\lm_i+\psi$ for all $1\le i\le n_1$, we say that $\{\lm_i\}$ and $\{\mu_i\}$ are {\em linearly related}.
\end{definition}

As we saw, we have

\begin{lemma}
\label{lemma:linearly-related}
If the frame field $\{v_i\}$ is irreducible, then $\{\lm_i\}$ and $\{\mu_i\}$ are linearly related.
\end{lemma}

When $n_1=3$, the frame field $\{v_i\}$ is reducible if and only if all
$g(\cd_{v_1}v_2,v_3)=g(\cd_{v_2}v_3,v_1)=g(\cd_{v_3}v_1,v_2)=0$.

\begin{proposition}
\label{proposition:vp,psi-constF}
If $\{\lm_i\}$ and $\{\mu_i\}$ have linear relation $\mu_i=\vp\lm_i+\psi$,
and if there exist at least three distinct $\lm_i$,
then $\vp$ and $\psi$ are constant on $F$.
\end{proposition}

\begin{proof}
For $i\ne j$, we have
\eqref{eq:v_i[lm_j]}
$v_i[\lm_j]=(\lm_i-\lm_j)g(\wt\cd_{v_j}v_i,v_j)$,
\eqref{eq:v_i[mu_j]}
$v_i[\mu_j]=(\mu_i-\mu_j)g(\wt\cd_{v_j}v_i,v_j)$.
Since $\mu_i=\vp\lm_i+\psi$, we obtain that
\begin{equation}
\begin{aligned}
v_i[\vp]\,\lm_j+\vp\,v_i[\lm_j]+v_i[\psi]
=\vp\,(\lm_i-\lm_j)\,g(\wt\cd_{v_j}v_i,v_j)
=\vp\,v_i[\lm_j].
\end{aligned}
\end{equation}
Thus, we have $v_i[\vp]\lm_j+v_i[\psi]=0$ ($i\ne j$).
Since we assume the existence of three different
$\lm_j$, we can conclude that $v_i[\vp]=v_i[\psi]=0$.
Therefore, $\vp$ and $\psi$ are constant on $F$.
\end{proof}

\section{Constantness of principal curvatures}
\label{section:constness-of-lmi}

In this section, we assume that $M^n$ is a non-minimal biharmonic hypersurface in the Euclidean space $E^{n+1}$.

In the following,
we assume that $\{\lm_i(t)\}$ and $\{\mu_i(t)\}$ have linear relation
$\mu_i(t)=\vp(t)\lm_i(t)+\psi(t)$.
This assumption holds,
under the case $n_1=2$ or
the condition that $\{v_i\}$ is irreducible in the case $n_1\ge3$.

In this section, we do not assume that
all the $\lm_i$ are simple.
However, by Lemma \ref{lemma:same-lm}, for the solutions other than
the one satisfying that $\tau\equiv0$, there exist at least two
$\lm_i$,
so $\vp$ and $\psi$ are uniquely determined.

In the following, we will assume that $\tau\ne0$,
and treat the solutions having
$\lm_i$ different each other.

\begin{lemma}
\label{lemma:dvp,dpsi}
The functions $\vp,\psi$ must satisfy the following two ordinary differential equations:
\begin{equation}
\begin{aligned}
\label{eq:dvp,dpsi}
&\vp'=-\frac12\tau(\vp^2+1)+\vp\psi,\quad
\psi'=(\psi-\frac12\tau\vp)\psi.
\end{aligned}
\end{equation}
\end{lemma}

\begin{proof}
By substituting $\mu_i=\vp\lm_i+\psi$ and \eqref{eq:ODEmu}
into \eqref{eq:ODElm}, we have
\begin{equation}
(\psi'-\psi^2+\frac12\tau\vp\psi)+(\vp'+\frac12\tau(\vp^2+1)-\psi\vp)\lm_i=0.
\end{equation}
Due to Lemma \ref{lemma:same-lm},
there exist $\lm_i$ which are different each other, and
we have \eqref{eq:dvp,dpsi} for $\vp',\psi'$.
\end{proof}

Later on, we will proceed calculations dividing by
$\vp$,
we first have to show in the following Lemma, which enables us to assume
that $\vp\ne0$.

\begin{lemma}
\label{lemma:vp!=0}
The function $\vp$ is not $0$ at a generic point.
\end{lemma}

\begin{proof}
Assume that $\vp\equiv0$.
By substituting this into the first equation of
Lemma \ref{lemma:dvp,dpsi},
we obtain $\tau=0$. We get Lemma \ref{lemma:vp!=0}.
\end{proof}

By Lemma \ref{lemma:vp!=0}, we will always assume that $\vp\not=0$.
Furthermore, we will have several lemmas for later uses.
We define the function $\Lm_k:=\ssum_{i\le n_1}(\lm_i)^k$ ($k=0,1,\cdots$).
Note that $\Lm_0=n_1$ by definition.

\begin{lemma}
\label{lemma:lm_i',tau'}
The differentiations
$\lm_i'$ and $\tau'$ can be expressed in terms of
$\vp$, $\psi$, and $\lm_i$ as follows:
\begin{equation}
\label{eq:lm_i',tau'}
\begin{aligned}
\lm_i'&=\frac12\tau\psi+(\psi+\frac12\tau\vp)\lm_i+\vp(\lm_i)^2,\\
\tau'&=\frac13(n_1+3)\tau\psi+\frac12\tau^2\vp+\frac23\vp\Lm_2.
\end{aligned}
\end{equation}
\end{lemma}

\begin{proof}
The first equation of \eqref{eq:lm_i',tau'}
can be obtained by substituting $\mu_i=\vp\lm_i+\psi$
simply into \eqref{eq:ODElm}, and we get the second one by summing it up.
\end{proof}

\begin{lemma}
\label{lemma:dLm}
The functions $\Lm_k$ satisfy the following ordinary differential equations.
\begin{equation}
\label{eq:dLm}
\Lm_k'
=k\{\frac12\tau\psi\Lm_{k-1}+(\frac12\tau\vp+\psi)\Lm_k+\vp\Lm_{k+1}\}.
\end{equation}
\end{lemma}

\begin{proof}
Lemma \ref{lemma:lm_i',tau'} implies that
\begin{equation}
\begin{aligned}
\Lm_k'&=\ssum((\lm_i)^k)'
=k\ssum(\lm_i)^{k-1}\lm_i'\\
&=k\ssum(\lm_i)^{k-1}
	\{\frac12\tau\psi+(\psi+\frac12\tau\vp)\lm_i+\vp(\lm_i)^2\}\\
&=k\ssum\{\frac12\tau\psi(\lm_i)^{k-1}
	+(\frac12\tau\vp+\psi)(\lm_i)^{k}+\vp(\lm_i)^{k+1}\},
\end{aligned}
\end{equation}
from which we obtain immediately \eqref{eq:dLm}.
\end{proof}

\begin{lemma}
\label{lemma:dtau}
The function $\tau$ satisfies the following ordinary differential equation:
\begin{equation}
\label{eq:dtau}
\begin{aligned}
\tau''
-(n_1\psi+\frac{3\tau(\vp^2+1)}{2\vp})\tau'
+\frac{\tau^2(\tau\vp+(n_1+3)\psi)}{2\vp}
=0.
\end{aligned}
\end{equation}
\end{lemma}

\begin{proof}
Differentiate
$\Lm_1=(3/2)\tau$ in $t$,
and apply Lemma \ref{lemma:dLm}, and
express $\Lm_1$ in terms of $\tau$, we have
\begin{equation}
\Lm_2
=\frac{6\tau'-3\tau^2\vp-2(n_1+3)\tau\psi}{4\vp},
\end{equation}
and $\ssum\mu_i=\vp\Lm_1+n_1\psi$.
Substituting these into $\ssum\mu_i$, $\ssum(\lm_i)^2$ of
\eqref{eq:ODEtau}, we obtain \eqref{eq:dtau}.
\end{proof}

\begin{lemma}
\label{lemma:vp!=const}
The function $\vp$ is not constant.
\end{lemma}

\begin{proof}
By Lemma \ref{lemma:vp!=0}, $\vp\ne0$.
Substituting $\vp=c\;(\ne0)$ \eqref{eq:dvp,dpsi} in
Lemma \ref{lemma:dvp,dpsi}, we have
$2c\psi-(1+c^2)\tau=0$, $\psi'=\psi^2-(1/2)c\tau\psi$,
thus we obtain $\tau'=(2c)^{-1}\tau^2$.
Substituting it and its differentiation into
\eqref{eq:dtau} in Lemma \ref{lemma:dtau}, we have
$(1+c^2)\tau^3=0$, which completes the proof.
\end{proof}

Next, we will show that $\psi$ is not constant if $\tau\ne0$.
To do it, we first show that $\psi\ne0$.
Assume that $\psi(t_1)=0$ at some point $t_1$
for the solution $(\lm_i,\mu_i,\vp,\psi)$.
Then, Lemma \ref{lemma:dvp,dpsi} and the uniqueness of solution to the ordinary differential equation
imply that $\psi(t)\equiv0$.
Then, it holds that $\mu_i=\vp\,\lm_i$ for all
$i$.
All the angles $\te_i=\arctan(\lm_i/\mu_i)=\arctan(1/\vp)$ in
Proposition \ref{proposition:r-theta} are equal to each other,
so we may write them as $\te$.
Then, if we write by $p$, one of indefinite integrals
$\int\cos\te\,dt$,
we have that $\vp=\cot\te$, $p'=\cos\te$, and
\begin{equation}
\begin{aligned}
&\te=\frac12\int\tau\,dt,\quad
r_i=\frac{-1}{p+b_i},\\
&\mu_i=\frac{-\cos\te}{p+b_i},\quad
\lm_i=\frac{-\sin\te}{p+b_i},\quad
\tau=-\frac23\sin\te\cdot\ssum_{i\le n_1}\frac{1}{p+b_i}.
\end{aligned}
\end{equation}

We define shortly $s_k:=\ssum_i(p+b_i)^{-k}$.
Since $\tau\ne0$, $p$ is not constant,
and we can define $s_k$ for generic $t$.
Then, substituting these into
\eqref{eq:P0} in Lemma \ref{lemma:P0}, we have
\begin{equation}
\label{eq:cos(te)-si}
\begin{aligned}
\cos^2\te=\frac{(s_1)^3+6s_1s_2}{3(2s_1s_2+3s_3)}.
\end{aligned}
\end{equation}
Differentiate this and substitute $p'=\cos\te$ and $\te'=\tau/2$, we have
\begin{equation}
\label{eq:sin(te)-si}
\begin{aligned}
&2s_1\sin^2\te\\
&\quad=-(2s_1s_2+3s_3)^{-2}
\{4(s_1)^3(s_2)^2-4(s_1)^4s_3+9(s_1)^2s_2s_3\\
&\qquad\qquad\qquad+18(s_2)^2s_3+36s_1(s_3)^2-9(s_1)^3s_4-54s_1s_2s_4\}.
\end{aligned}
\end{equation}
By eliminating $\cos\te$ and $\sin\te$ from \eqref{eq:cos(te)-si} and \eqref{eq:sin(te)-si},
we have
\begin{equation}
\begin{aligned}
Q(p)&:=-4(s_1)^5s_2+12(s_1)^3(s_2)^2-18(s_1)^4s_3+63(s_1)^2s_2s_3\\
&\qquad+54(s_2)^2s_3+162s_1(s_3)^2-27(s_1)^3s_4-162s_1s_2s_4\\
&=0.
\end{aligned}
\end{equation}

Since $\tau\ne0$, $p$ is not constant.
Thus, $Q(p)$ must vanish identically as a rational function in $p$.
Thus, we obtain in particular,
$\=\lim_{p\to\infty}p^7Q(p)\allowbreak=0$.
On the other hand, since $\=\lim_{p\to\infty}p^ks_k=n_1$ for each
$s_k$, we obtain that
$0=-2(n_1)^3(n_1-3)(n_1+3)(2n_1+3)$.
Thus we have that $n_1=3$.

Next, denote by $m$, the number of $b_i$ which are equal to $b_1$.
Let us consider the coefficients of $(p+b_1)^{-7}$ in the partial fraction decomposition of $Q(p)$.
Then, since the coefficients
are equal to the one exchanging each $s_k$ into
$m$ in $Q(p)$, the coefficient is equal to
$-2m^3(m-3)(m+3)(2m+3)$, and it must vanish.
Thus, we obtain that $m=3$.

It means that all the $b_i$ are equal to each other.
But, in this case, all the $\lm_i$ must be equal to each other, and due to Lemma \ref{lemma:same-lm}, we obtain
that $\tau\equiv0$.

Therefore, we obtain the following lemma.
\begin{lemma}
\label{lemma:psi!=0}
The function $\psi$ does not attain $0$.
\end{lemma}

Finally, we can show that $\psi$ is not constant by Lemma \ref{lemma:psi!=0}.

\begin{lemma}
\label{lemma:psi!=constant}
The function $\psi$ is not constant.
\end{lemma}

\begin{proof}
Assume that $\psi=c\ne0$.
Then, by substituting it into \eqref{eq:dvp,dpsi} in
Lemma \ref{lemma:dvp,dpsi}, we have
$\vp'=(1/2)(2c\vp-\tau(1+\vp^2))$, and
$c(2c-\tau\vp)=0$.
Then, we have $\vp=2c/\tau$, $\tau'=\tau^3/(4c)$.
Together it and its differentiation,
and Lemma \ref{lemma:dtau}, we obtain $\tau^3=0$. Thus, we have
$\tau\equiv 0$.
\end{proof}

\begin{proposition}
\label{proposition:const-lmmu}
Assume that
$n_1\ge3$,
all the principal curvatures $\lambda_i$ are simple, and
$\tau\ne0$. Then, all the
$\tau$, $\vp$, $\psi$, $\lm_i$, and $\mu_i$
must be constant on $F$.
\end{proposition}

\begin{proof}
In the following, we will show inductively
that $\Lm_k$ are constant along each $F$.
$\Lm_0=n_1$, and $\Lm_1=(3/2)\tau$ are constant
along each $F$.
Assume that all the
$\Lm_\ell$ with $\ell\le k$, are constant along each
$F$.
Then, their differentiations with respect to
$t$ are also constant along each $F$.
Therefore, by
Lemma \ref{lemma:dLm},
\begin{equation}
\frac12\tau\psi\Lm_{k-1}+(\frac12\tau\vp+\psi)\Lm_k+\vp\Lm_{k+1}
\end{equation}
are also constant along each $F$.
Since $\vp$ and $\psi$ are constant along each $F$
by Proposition \ref{proposition:vp,psi-constF}, and $\vp\ne0$,
$\Lm_{k+1}$ is also constant along each $F$.

Therefore,
all the elementary symmetric polynomials in $\lm_i$
are constant, and every $\lm_i$ is also constant along each $F$.
\end{proof}

\section{Proof of Main theorem}
\label{section:proof-of-main-theorem}

In this section, we assume that $M^n$ is a non-minimal biharmonic hypersurface in the Euclidean space $E^{n+1}$, all the principal curvatures of $M$ are simple, and the frame field is irreducible.

For every distinct triplet $\{i,j,k\}$,
\eqref{eq:v_i[lm_j]} and \eqref{eq:s^lm} hold, i.e.,
\begin{equation}
\begin{aligned}
&v_i[\lm_j]=(\lm_i-\lm_j)g(\wt\cd_{v_j}v_i,v_j),\\
&s^\lm_{ijk}:=(\lm_j-\lm_k)g(\wt\cd_{v_i}v_j,v_k)
\text{ are symmetric in $i,j$ and $k$}.
\end{aligned}
\end{equation}
By Proposition \ref{proposition:const-lmmu}, $\lm_j$ are constant along each $F$, which imply that $v_i[\lm_j]=0$.
Thus, since
$\lm_i$ are simple,
it holds that $g(\wt\cd_{v_j}v_i,v_j)=0$.
Therefore, $g(\wt\cd_{v_j}v_j,v_i)=-g(\wt\cd_{v_j}v_i,v_j)=0$.
Combining with $g(\wt\cd_{v_j}v_j,v_j)=0$, we get $\wt\cd_{v_j}v_j=0$.

By the definition of the curvature tensor field and
$g(\wt\cd_{v_i}v_j,v_k)=s^\lm_{ijk}/(\lm_j-\lm_k)$,
we obtain
\begin{equation}
\begin{aligned}
&g(\wt R(v_i,v_j)v_j,v_i)
=g(\wt\cd_{v_i}\wt\cd_{v_j}v_j-\wt\cd_{v_j}\wt\cd_{v_i}v_j
	-\wt\cd_{[v_i,v_j]}v_j,v_i)\\
&=-v_j[g(\wt\cd_{v_i}v_j,v_i)]+g(\wt\cd_{v_i}v_j,\wt\cd_{v_j}v_i)
	-g(\wt\cd_{\wt\cd_{v_i}v_j-\wt\cd_{v_j}v_i}v_j,v_i)\\
&	\qquad\qquad(\text{by }\wt\cd_{v_j}v_j=0)\\
&=\ssum_kg(\wt\cd_{v_i}v_j,v_k)g(\wt\cd_{v_j}v_i,v_k)\\
&\qquad	-\ssum_k(g(\wt\cd_{v_i}v_j,v_k)
	-g(\wt\cd_{v_j}v_i,v_k))g(\wt\cd_{v_k}v_j,v_i)\\
&	\qquad\qquad(\text{by }g(\wt\cd_{v_i}v_j,v_i)=0)\\
&=\sum_{k\ne i,j}\frac{s^\lm_{ijk}}{\lm_j-\lm_k}
		\cdot\frac{s^\lm_{jik}}{\lm_i-\lm_k}
	-\sum_{k\ne i,j}
		(\frac{s^\lm_{ijk}}{\lm_j-\lm_k}-\frac{s^\lm_{jik}}{\lm_i-\lm_k})
		\cdot\frac{s^\lm_{kji}}{\lm_j-\lm_i}\\
&=\sum_{k\ne i,j}\frac{2(s^\lm_{ijk})^2}{(\lm_i-\lm_k)(\lm_j-\lm_k)}.
\end{aligned}
\end{equation}

Therefore, the scalar curvature $\wt s$ of $F$ is expressed as
\begin{equation}
\label{eq:Rijji=slm}
\wt s
=\sum_{i\ne j}g(\wt R(v_i,v_j)v_j,v_i)
=2\sum_{i<j,\;k\ne i,j}\frac{2(s^\lm_{ijk})^2}{(\lm_i-\lm_k)(\lm_j-\lm_k)}.
\end{equation}
Note that
the terms of the right hand side of \eqref{eq:Rijji=slm} having the same
$s^\lm_{i_1i_2i_3}$ ($i_1<i_2<i_3$) are
$(i,j,k)=(i_1,i_2,i_3),(i_1,i_3,i_2),(i_2,i_3,i_1)$,
and by a simple computation, their sum vanishes
as follows:
\begin{equation}
\begin{aligned}
&\frac{2(s^\lm_{i_1i_2i_3})^2}{(\lm_{i_1}-\lm_{i_3})(\lm_{i_2}-\lm_{i_3})}
+\frac{2(s^\lm_{i_1i_2i_3})^2}{(\lm_{i_1}-\lm_{i_2})(\lm_{i_3}-\lm_{i_2})}
+\frac{2(s^\lm_{i_1i_2i_3})^2}{(\lm_{i_2}-\lm_{i_1})(\lm_{i_3}-\lm_{i_1})}\\
&\qquad=0.
\end{aligned}
\end{equation}

Thus, we have the following

\begin{lemma}
Every $F$ has zero scalar curvature.
\end{lemma}

On the other hand, applying Gauss equation to $F$ regarding as a submanifold of
$E^{n_1+2}$,
we have, for every $i\ne j$,
\begin{equation}
\begin{aligned}
g(\wt R(v_i,v_j)v_j,v_i)
&=\ol g(\wt\al(v_i,v_i),\wt\al(v_j,v_j))
	-\ol g(\wt\al(v_i,v_j),\wt\al(v_i,v_j))\\
&=\lm_i\lm_j+\mu_i\mu_j.
\end{aligned}
\end{equation}

Thus, we have
\begin{equation}
\label{eq:Lm2-tau-vp-phi}
\begin{aligned}
0&=\ssum_{i\ne j}\;g(\wt R(v_i,v_j)v_j,v_i)
=\ssum_{i<j}\;(2\lm_i\lm_j+2\mu_i\mu_j)\\
&=(\ssum_i\lm_i)^2-\ssum_i(\lm_i)^2+(\ssum_i\mu_i)^2-\ssum_i(\mu_i)^2\\
&=(n_1-1)\psi(n_1\psi+3\tau\vp)
+\frac14(1+\vp^2)(9\tau^2-4\Lm_2).
\end{aligned}
\end{equation}

We rewrite \eqref{eq:Lm2-tau-vp-phi} as follows.
\begin{equation}
\label{eq:Lm20}
\begin{aligned}
4(n_1-1)\psi(n_1\psi+3\tau\vp)+(1+\vp^2)(9\tau^2-4\Lm_2)
=0.
\end{aligned}
\end{equation}
By substituting $\mu_i=\vp\lm_i+\psi$ into \eqref{eq:P0} of
Lemma \ref{lemma:P0}, we have
\begin{equation}
\label{eq:Lm30}
\begin{aligned}
4&\bigl((n_1)^2-9\bigr)\tau\psi^2
+2\vp\psi\{(6n_1-9)\tau^2+4(n_1-6)\Lm_2\}\\
&\qquad+3\{3\tau^3(1+\vp^2)+8\tau\Lm_2-8\vp^2\Lm_3\}\\
&=0
\end{aligned}
\end{equation}
Together with \eqref{eq:Lm20} and \eqref{eq:Lm30},
and differentiating twice \eqref{eq:Lm20} and \eqref{eq:Lm30}, and eliminating
$\Lm_3,\Lm_2$, and $\tau$,
we obtain our main theorem.

\begin{theorem}
\label{theorem:bh-is-minimal}
Every biharmonic hypersurface $M$ in the Euclidean space is minimal if we assume that
all the principal curvatures are simple, and the frame $\{v_i\}$
of a characteristic submanifold $F=\{\tau=c\}$
for a constant $c$, is irreducible.
\end{theorem}

\begin{proof}
In the following, we will proceed to eliminate
$\Lm_3,\Lm_2$, and $\tau$ exactly.
By differentiating
\eqref{eq:Lm20} in $t$, we have
\begin{equation}
\label{eq:Lm3}
\begin{aligned}
&9(n_1-1)\tau\vp\psi^2+2n_1(n_1-1)\psi^3\\
&\quad+\frac{1}{2}\psi\{9\tau^2(1+2\vp^2)+4\Lm_2(-1+(n_1-3)\vp^2)\}\\
&\quad+\vp(1+\vp^2)(3\tau\Lm_2-2\Lm_3)\\
&=0.
\end{aligned}
\end{equation}

Eliminating $\Lm_3$ from
\eqref{eq:Lm30} and \eqref{eq:Lm3}, we have
\begin{equation}
\label{eq:Lm2}
\begin{aligned}
&-9\tau^3(1+\vp^2)^2
-6\tau^2\vp\psi\{2(n_1-6)+(2n_1-21)\vp^2\}\\
&\quad-4\tau\psi^2\{(n_1)^2-9+((n_1)^2-27n_1+18)\vp^2\}\\
&\quad+24n_1(n_1-1)\vp\psi^3\\
&\quad+4\Lm_2\{3\tau(-2+\vp^2+3\vp^4)+2\vp(-n_1+3+(2n_1-3)\vp^2)\psi\}\\
&=0.
\end{aligned}
\end{equation}

Eliminating $\Lm_2$ from
\eqref{eq:Lm20} and \eqref{eq:Lm2}, we have
\begin{equation}
\label{eq:taup3}
\begin{aligned}
&9\tau^3(1+\vp^2)^2(-7+8\vp^2)\\
&\quad+6\tau^2\vp\psi(1+\vp^2)(-17n_1+33+2(11n_1-3)\vp^2)\\
&\quad+4\tau\psi^2\{-7(n_1)^2+6n_1+9-(n_1-9)(5n_1-3)\vp^2\\
&\qquad\qquad\qquad+4n_1(5n_1-3)\vp^4\}\\
&\quad+8n_1(n_1-1)\vp\psi^3(-n_1+6+2n_1\vp^2)\\
&=0.
\end{aligned}
\end{equation}

To eliminate $\tau$,
differentiate
\eqref{eq:taup3} in $t$, and apply Lemma \ref{lemma:lm_i',tau'} to $\tau'$, and eliminate $\Lm_2$
by using \eqref{eq:Lm20}. Then, we obtain the following four equations:

By differentiating the first term of \eqref{eq:taup3}, we have
\begin{equation}
\label{eq:bh-is-minimal-part1}
\begin{aligned}
\{&9\tau^3(1+\vp^2)^2(-7+8\vp^2)\}'\\
&=9\tau^2(1+\vp^2)(-7+8\vp^2)\\
&\qquad\times\{6\tau^2\vp(1+\vp^2)+\tau\psi(n_1+3+(7n_1-3)\vp^2)\\
&\qquad\qquad+2n_1(n_1-1)\vp\psi^2\}\\
&\quad-54\tau^3\vp(1+\vp^2)(-1+4\vp^2)(\tau(1+\vp^2)-2\vp\psi).
\end{aligned}
\end{equation}

By differentiating the second term of \eqref{eq:taup3}, we have
\begin{equation}
\label{eq:bh-is-minimal-part2}
\begin{aligned}
&\{6\tau^2\vp\psi(1+\vp^2)\bigl(-17n_1+33+2(11n_1-3)\vp^2\bigr)\}'\\
&=\tau\vp\psi\{-17n_1+33+2(11n_1-3)\vp^2\}\\
&\qquad\times
\{21\tau^2\vp(1+\vp^2)+2\tau\psi(2n_1+9+(14n_1-3)\vp^2)\\
&\qquad\qquad+8n_1(n_1-1)\vp\psi^2\}\\
&\quad-3\tau^2\psi(\tau(1+\vp^2)-2\vp\psi)\\
&\qquad\times\{-17n_1+33+3(5n_1+27)\vp^2+10(11n_1-3)\vp^4\}.
\end{aligned}
\end{equation}

By differentiating the third term of \eqref{eq:taup3}, we have
\begin{equation}
\label{eq:bh-is-minimal-part3}
\begin{aligned}
&\{4\tau\psi^2p\}'\\
&=4\tau^2\vp\psi^2p
+\frac{8n_1(n_1-1)\vp\psi^4p}{3(1+\vp^2)}\\
&\quad+4\tau p
\{2\psi^3+\frac{1}{3}(n_1+3)\psi^3+\frac{2(n_1-1)\vp^2\psi^3}{1+\vp^2}\}\\
&\quad-4(5n_1-3)\tau\vp\psi^2(-n_1+9+8n_1\vp^2)(\tau(1+\vp^2)-2\vp\psi),
\end{aligned}
\end{equation}
where $p=-7(n_1)^2+6n_1+9-(n_1-9)(5n_1-3)\vp^2+4n_1(5n_1-3)\vp^4$.

And finally differentiating the fourth term of \eqref{eq:taup3}, we have
\begin{equation}
\label{eq:bh-is-minimal-part4}
\begin{aligned}
\{8&n_1(n_1-1)\vp\psi^3(-n_1+6+2n_1\vp^2)\}'\\
&=-4n_1(n_1-1)\psi^3(-n_1+6+2n_1\vp^2)(\tau(1+4\vp^2)-8\vp\psi)\\
&\quad-16(n_1)^2(n_1-1)\vp^2\psi^3(\tau(1+\vp^2)-2\vp\psi).
\end{aligned}
\end{equation}

Together \eqref{eq:bh-is-minimal-part1}, \eqref{eq:bh-is-minimal-part2}, \eqref{eq:bh-is-minimal-part3} and \eqref{eq:bh-is-minimal-part4}, by multiplying the denominator, and
setting all the terms of the resulting equation in order of $\tau$, we obtain the following equation:
\begin{equation}
\label{eq:taup4}
\begin{aligned}
&324\tau^4\vp(1+\vp^2)^3(-3+2\vp^2)\\
&\,+9\tau^3\psi(1+\vp^2)^2\\
&\qquad\times\{-4(n_1+24)-(257n_1-249)\vp^2+4(53n_1+15)\vp^4\}\\
&\,+6\tau^2\vp\psi^2(1+\vp^2)\bigl\{-101(n_1)^2-159n_1+468\\
&\qquad\qquad-(265(n_1)^2-783n_1-90)\vp^2
+4(85(n_1)^2+33n_1-18)\vp^4\bigr\}\\
&\,+4\tau\psi^3\bigl\{-4(n_1)^3-78(n_1)^2+81n_1+81\\
&\qquad\qquad\qquad-3(53(n_1)^3-69(n_1)^2-222n_1+126)\vp^2\\
&\qquad\qquad\qquad-9(3(n_1)^3-89(n_1)^2+11n_1+27)\vp^4\\
&\qquad\qquad\qquad+4n_1(59(n_1)^2+21n_1-36)\vp^6\bigr\}\\
&\,+8n_1(n_1-1)\vp\psi^4\bigl\{-(n_1-3)(7n_1+27)\\
&\qquad\qquad\qquad\qquad\qquad-(n_1-15)(5n_1+3)\vp^2
+4n_1(5n_1+6)\vp^4\bigr\}\\
&=0.
\end{aligned}
\end{equation}

The equation \eqref{eq:taup3}
is of third order in $\tau$, and \eqref{eq:taup4} is of fourth order in $\tau$.
Therefore, we use Euclid's algorithm to eliminate
$\tau$.
Namely,
dividing \eqref{eq:taup4} by \eqref{eq:taup3},
the remainder, denoted by
$f_2$, is of order two in $\tau$.
And dividing
\eqref{eq:taup3} by $f_2$, we denote the remainder by
$f_1$, and
finally dividing $f_2$ by $f_1$,
we denote the remainder by $f_0$, then $f_0$ does not
include $\tau$.
Note that the remainder is of the form
of rational function $r/q$ (where $q,r$ are polynomials).
But, if we multiply by $q$ in advance, we may ignore $q$, and may use $r$ in the next step.
Finally the obtained
numerator is of the form which is the multiplication
of the following three polynomials:
\begin{equation}
\label{eq:taup0-a}
72n_1(n_1-1)\vp^3\psi^4(1+\vp^2)^2,
\end{equation}
\begin{equation}
\label{eq:taup0-b}
\begin{aligned}
&\{-105(3(n_1)^2+n_1+12)+(1026(n_1)^2+875n_1-5901)\vp^2\\
&\quad-2(567(n_1)^2+1264n_1-351)\vp^4
+8(45(n_1)^2+34n_1-159)\vp^6\}^2,
\end{aligned}
\end{equation}
\begin{equation}
\begin{aligned}
\label{eq:taup0-c}
-&7(17n_1-33)^2(5(n_1)^4-6(n_1)^3-13(n_1)^2+3n_1+27)\\
&-\bigl\{100964(n_1)^6-269159(n_1)^5-362329(n_1)^4+1289439(n_1)^3\\
&\qquad\quad+839475(n_1)^2-4210164n_1+2755134\bigr\}\vp^2\\
&+3\bigl\{167725(n_1)^6-789504(n_1)^5+904142(n_1)^4+651168(n_1)^3\\
&\qquad\quad-3766311(n_1)^2+5957928n_1-3211164\bigr\}\vp^4\\
&-\bigl\{825166(n_1)^6-5345493(n_1)^5+11702488(n_1)^4+1018458(n_1)^3\\
&\qquad\quad-21744558(n_1)^2+11059011n_1+2628288\bigr\}\vp^6\\
&+\bigl\{576422(n_1)^6-5693096(n_1)^5+13627127(n_1)^4-9103320(n_1)^3\\
&\qquad\quad-22777452(n_1)^2+44245224n_1-20731545\bigr\}\vp^8\\
&-12\bigl\{13322(n_1)^6-199503(n_1)^5+758200(n_1)^4-367722(n_1)^3\\
&\qquad\quad-1262466(n_1)^2+1720305n_1-683640\bigr\}\vp^{10}\\
&-4\bigl\{1094(n_1)^6+56460(n_1)^5-446779(n_1)^4+834336(n_1)^3\\
&\qquad\quad+731808(n_1)^2-2056212n_1+843453\bigr\}\vp^{12}\\
&+16(7n_1-9)^2(2(n_1)^2-9)(5(n_1)^2-26n_1-43)\vp^{14}.
\end{aligned}
\end{equation}

The multiplication $f_0$ of three equations
\eqref{eq:taup0-a}, \eqref{eq:taup0-b} and \eqref{eq:taup0-c}
should be zero identically.
This implies that one of the factors of $f_0$ should be zero identically.
Note that each factor of $f_0$ is a polynomial only in
$\vp$ or $\psi$, and its coefficient
of the highest term is
non-zero for every natural number $n_1$.
Therefore, if $f_0$ vanishes identically,
then $\vp$ or $\psi$ must be a constant.
Thus, by
Lemmas \ref{lemma:vp!=const} and \ref{lemma:psi!=constant}, we obtain that $\tau\equiv0$.

We have done.
\end{proof}

Since Euclid's algorithm for polynomials is a tedious calculation, we will give examples of calculation using a computer in Appendix.

\section{Appendix: Euclid's algorithm using a computer}

We give two examples of calculations using formula manipulation systems, mathematica\footnote{Mathematica is a registered trademark of Wolfram Research Inc.}
and Maple\footnote{Maple is a registered trademark of Waterloo Maple Inc.}.
Both are commercial softwares, but free softwares probably have similar functions.

\medbreak
With mathematica, we calculate as follows.

\medbreak
\noindent
\verb|f3 = |$\langle$the left hand side of \eqref{eq:taup3}$\rangle$\verb|;|\\
\verb|f4 = |$\langle$the left hand side of \eqref{eq:taup4}$\rangle$\verb|;|\\
\verb|f2 = Numerator[Factor[PolynomialRemainder[f4, f3, tau]]];|\\
\verb|f1 = Numerator[Factor[PolynomialRemainder[f3, f2, tau]]];|\\
\verb|f0 = Numerator[Factor[PolynomialRemainder[f2, f1, tau]]]|
\smallbreak

\medbreak
For Maple, we prepare a function which calculates remainder of multivariable polynomials.

\begin{verbatim}
with(PolynomialTools);
polynomialremainder := proc(poly1, poly2, var)
  local cf1, cf2, deg11, deg12, top2, ratio, i, j, poly;
  cf1 := CoefficientList(poly1, var);
  deg11 := numelems(cf1);
  cf2 := CoefficientList(poly2, var);
  deg12 := numelems(cf2);
  top2 := cf2[deg12];
  for j from deg11 by -1 to deg12 do
    ratio := cf1[j]/top2;
    for i from 0 to deg12-1 do
      cf1[j-i] := cf1[j-i]-cf2[deg12-i]*ratio
    end do
  end do;
  poly := 0;
  for i to deg11 do
    poly := poly+cf1[i]*var^(i-1)
  end do;
  return poly
end proc;
\end{verbatim}

Using this, we calculate as follows.

\smallbreak
\noindent
\verb|f3 := |$\langle$the left hand side of \eqref{eq:taup3}$\rangle$\verb|;|\\
\verb|f4 := |$\langle$the left hand side of \eqref{eq:taup4}$\rangle$\verb|;|\\
\verb|f2 := numer(factor(polynomialremainder(f4, f3, tau)));|\\
\verb|f1 := numer(factor(polynomialremainder(f3, f2, tau)));|\\
\verb|f0 := numer(factor(polynomialremainder(f2, f1, tau)));|
\smallbreak

\end{document}